\documentclass[11pt]{article}
\usepackage{amsmath, amsthm, amssymb,latexsym, enumerate,kotex}
\usepackage{graphicx}
\usepackage{hyperref}
\usepackage{xcolor}
\usepackage{cite}
\usepackage{dsfont}
\usepackage{graphicx, caption, subcaption}
\usepackage{subcaption}
\usepackage{caption}
\usepackage{colortbl}
\usepackage{float}
\usepackage{ulem}
\usepackage{tikz}
\numberwithin{equation}{section}
\usepackage{array}
\newcolumntype{L}[1]{>{\raggedright\let\newline\\\arraybackslash\hspace{0pt}}m{#1}}
\newcolumntype{C}[1]{>{\centering\let\newline\\\arraybackslash\hspace{0pt}}m{#1}}
\newcolumntype{R}[1]{>{\raggedleft\let\newline\\\arraybackslash\hspace{0pt}}m{#1}}

\newtheorem{theorem}{Theorem}[section]
\newtheorem{lemma}[theorem]{Lemma}
\newtheorem{fact}[theorem]{Fact}

\newtheorem{conjecture}[theorem]{Conjecture}
\newtheorem{question}[theorem]{Question}

\newtheorem{definition}[theorem]{Definition}

\title{Connectivity keeping trees in triangle-free graphs} 

\author{Hojin Chu\thanks{Korea Institute for Advanced Study, Seoul 02455, Republic of Korea. Email: hojinchu@kias.re.kr},
\and Shinya Fujita,\thanks{School of Data Science, Yokohama City University, Yokohama 236-0027, Japan.  Email: fujita@yokohama-cu.ac.jp}
\and Boram Park\thanks{Department of Mathematics Education, Seoul National University, Seoul 08826, Republic of Korea. Email:borampark@snu.ac.kr},
\and Homoon Ryu\thanks{Department of Mathematics, Ajou University, Suwon 16499, Republic of Korea. Email: ryuhomoon@ajou.ac.kr}
}

\date{}

\begin{document}

\maketitle

\begin{abstract}
In 2010, Mader conjectured that for any tree $T$ of order $m$, every $k$-connected graph $G$ with minimum degree at least $\lfloor \frac{3k}{2}\rfloor+m-1$ contains a subtree $T'\cong T$ such that $G-V(T')$ remains $k$-connected.
In 2022, Luo, Tian, and Wu considered an analogous problem for bipartite graphs and conjectured that for any tree $T$ with bipartition $(X,Y)$, every $k$-connected bipartite graph $G$ with minimum degree at least $k+\max\{|X|,|Y|\}$ contains a subtree $T'\cong T$ such that $G-V(T')$ remains $k$-connected.
In this paper, we relax the bipartite assumption by considering triangle-free graphs and prove that for any tree $T$ of order $m$, every $k$-connected triangle-free graph $G$ with minimum degree at least $2k+3m-4$ contains a subtree $T' \cong T$ such that $G-V(T')$ remains $k$-connected. Furthermore, we establish refined results for specific subclasses such as bipartite graphs or graphs with girth at least five. 
\end{abstract}

\bigskip



\section{Introduction}
Throughout this paper, all graphs are finite, simple, and undirected. The {\it order} of a graph is the cardinality of its vertex set. For a vertex set $W\subseteq V(G)$ of a graph $G$, $G[W]$ is the subgraph induced by $W$ and $G-W = G[V(G)\setminus W]$.
We denote by $N_G(W)$ the set of vertices in $V(G)\setminus W$ that have a neighbor in $W$, and write $N_G(v)$ for $N_G(\{v\})$.
For a graph $G$, we denote the {\it minimum degree}, the {\it maximum degree}, and the {\it connectivity} of $G$ by $\delta(G)$, $\Delta(G)$, and $\kappa(G)$, respectively. 
We also denote by $g(G)$ the {\it girth} of $G$, that is, the length of its shortest cycle. 
For a positive integer $k$, if $G$ is a graph with $\kappa(G) \ge k$, then $G$ is called {\it $k$-connected}. Unless otherwise specified, $k$ and $m$ are always positive integers.

The problem of connectivity keeping subgraphs studies how large the minimum degree of a $k$-connected graph must be to ensure that the deletion of a specified subgraph does not destroy its $k$-connectivity. (See~\cite{TM26} for a survey.)
A fundamental result was established by Chartrand, Kaugars, and Lick~\cite{CKL72} in 1972.
\begin{theorem}[\!\!\cite{CKL72}]\label{thm:Chartrand}
Let $k$ be a positive integer.
Every $k$-connected graph $G$ with $\delta(G) \ge \lfloor \frac{3k}{2}\rfloor$ has a vertex $v$ such that $G-v$ remains $k$-connected.
\end{theorem}
In 2008, Fujita and Kawarabayashi~\cite{FK08} proved that every $k$-connected graph $G$ with $\delta(G) \ge \lfloor \frac{3k}{2}\rfloor+2$ contains an edge $e=uv$ such that $G-\{u,v\}$ remains $k$-connected.
They conjectured that given positive integers $k$ and $m$, there is a function $f_{k}(m)$ such that every $k$-connected graph $G$ with $\delta(G) \ge \lfloor \frac{3k}{2}\rfloor+f_k(m)$ contains a connected subgraph $H$ of order $m$ satisfying $\kappa(G-V(H))\ge k$.
Indeed, they gave examples to show that $f_k(m)$ is at least $m-1$.
In 2010, Mader solved the conjecture by proving that such a subgraph can be chosen to be a path as follows.

\begin{theorem}[\!\!\cite{Mader10}]\label{thm:mader}
Let $k$ and $m$ be positive integers.
Every $k$-connected graph $G$ with $\delta(G) \ge \lfloor \frac{3k}{2}\rfloor+m-1$ contains a path $P$ of order $m$  such that $\kappa(G-V(P)) \ge k$.
\end{theorem}

Mader further conjectured that Theorem~\ref{thm:mader} holds not only for a path of order $m$, but also any tree of order $m$.

\begin{conjecture}[\!\!\cite{Mader10}]\label{conj:mader}
Let $k$ and $m$ be positive integers.
For any tree $T$ of order $m$, every $k$-connected graph $G$ with $\delta(G) \ge \lfloor \frac{3k}{2} \rfloor + m -1$ contains a subtree $T' \cong T$ such that $\kappa(G-V(T')) \ge k$.
\end{conjecture}

In a follow-up study, Mader~\cite{Mader12} proved that, for any tree $T$ of order $m$ and any $k$-connected graph $G$ with $\delta(G) \ge 2(m+k-1)^2+m-1$, there exists a subtree $T'\cong T$ of $G$ such that $G-V(T')$ remains $k$-connected.
Diwan and Tholiya~\cite{DT09} proved that the conjecture is true for $k=1$, and Hong and Liu~\cite{HL22} proved that it is true for $k\le 3$. The conjecture remains open for $k\ge 4$. Recently, Liu, Ying, and Hong~\cite{LYH24} proved that Conjecture~\ref{conj:mader} holds for minimum degree up to linear in $k$ and $m$ as follows.

\begin{theorem}[\!\!\cite{LYH24}]\label{thm:LYH24}
Let $k$ and $m$ be positive integers.
For any tree $T$ of order $m$, every $k$-connected graph $G$ with $\delta(G) \ge  3k+4m-6$ contains a subtree $T' \cong T$ such that $G-V(T')$ remains $k$-connected.
\end{theorem}

The problem of connectivity keeping subgraphs has also been studied extensively in the context of bipartite graphs.
Luo, Tian, and Wu~\cite{LTW22} showed that the minimum degree condition of Theorem~\ref{thm:mader} can be relaxed under the bipartite condition.
\begin{theorem}[\!\!\cite{LTW22}]\label{thm:LTW22}
Let $k$ and $m$ be positive integers.
    Every $k$-connected bipartite graph $G$ with $\delta(G) \ge k+m$ contains a path $P$ of order $m$ such that $\kappa(G-V(P)) \ge k$.
\end{theorem}
They also proposed the following bipartite version of Conjecture~\ref{conj:mader}.

\begin{conjecture}[\!\!\cite{LTW22}]\label{conj:bip}
Let $k$ be a positive integer.
For any tree $T$ with bipartition $(X,Y)$, every $k$-connected bipartite graph $G$ with $\delta(G) \ge k+\max\{|X|,|Y|\}$ contains a subtree $T' \cong T$ such that $\kappa(G-V(T')) \ge k$.
\end{conjecture}

Conjecture~\ref{conj:bip} was studied for special classes of trees when $k$ has small value, and most recently, Hong, Wu, and Liu~\cite{HWL25} improved these results by confirming Conjecture~\ref{conj:bip} for $k \le 3$. The conjecture remains open for $k \ge 4$.
Yang and Tian~\cite{YT25} proved that Conjecture~\ref{conj:bip} holds for odd paths by showing that every $k$-connected bipartite graph $G$ with $\delta(G) \ge k + \frac{m+1}{2}$ contains a path $P$ of order $m$ such that $G-V(P)$ remains $k$-connected.
Recently, Fujita~\cite{Fujita25} relaxed the bipartite condition to the triangle-free condition  by showing that every $k$-connected triangle-free graph $G$ with $\delta(G) \ge k + \lceil \frac{m+1}{2}\rceil$ contains a path $P$ of order $m$ such that $G-V(P)$ remains $k$-connected. 

In this paper, we consider the existence of a connectivity keeping tree in a triangle-free graph as follows.

\begin{theorem}\label{thm:main_trianglefree}
Let $k$ and $m$ be positive integers, and $T$ be any tree of order $m$. Then every $k$-connected triangle-free graph $G$ with $\delta(G) \ge 2k + 3m -4$ contains a subtree $T' \cong T$ such that $\kappa(G-V(T')) \ge k$.
\end{theorem}

We also obtain refined bounds for subfamilies of triangle-free graphs, such as bipartite graphs or graphs with girth at least five, by an argument analogous to that used in the proof of Theorem~\ref{thm:main_trianglefree}.

\begin{theorem}\label{thm:main_bip_g5}
Let $k$ and $m$ be positive integers, and $T$ be any tree of order $m$ with bipartition $(X,Y)$. Then every $k$-connected graph $G$ contains a subtree $T' \cong T$ such that $\kappa(G-V(T')) \ge k$ if $G$ satisfies one of the following:
\begin{itemize}
\item[\rm(i)] $g(G)\ge 5$ and $\delta(G) \ge 2k + 2m+ \max\{ \frac{m-1}{2}, \Delta(T)\} -3$;
\item[\rm(ii)] $G$ is bipartite and $\delta(G) \ge 2k + 2m+ \max\{|X|,|Y|\} -3$.
\end{itemize}
\end{theorem}

While Conjecture~\ref{conj:bip} remains open, the best known previous result for bipartite graphs was obtained in Theorem~\ref{thm:LYH24}. Our result, Theorem~\ref{thm:main_bip_g5}~(ii), gives a further improvement. 

In the remainder of this paper, we prove the main results stated above. 
We first revisit a degree condition that guarantees the existence of a suitable subtree.
We use the notion of a connected triple, which is a structure that seems to have broad applications to graph connectivity problems. 
Our results also arise from an application of connected triples.

\section{Preliminaries}

\subsection{Tree embedding}

The following well-known lemma provides a sufficient minimum degree condition ensuring that every tree on $m$ vertices appears in a graph.

\begin{lemma}[\!\!\cite{Chvatal77}]\label{lem:tree_embedding}
   If $G$ is a graph of minimum degree at least $m-1$ for a positive integer $m$, then $G$ contains every tree on $m$ vertices. 
\end{lemma}

For bipartite host graphs, a tree can also be embedded under degree conditions that take its bipartition into account. The following is an immediate consequence of a result of Orchel~\cite{Orchel06}.
\begin{theorem}[\!\!\cite{Orchel06}]\label{thm:tree_embedding_bipartite}
Let $G$ be a bipartite graph and $T$ be a tree with bipartition $(X,Y)$.
If $\delta(G)\ge \max\{|X|,|Y|\}$, then there exists a subtree $T' \cong T$ of $G$.
\end{theorem}
For graphs of girth at least five, we use the following embedding result of Brandt and Dobson~\cite{BD96}.

\begin{theorem}[\!\!\cite{BD96}]\label{thm:g5}
Let $G$ be a graph with $g(G) \ge 5$ and $T$ be a tree of order $m$. If $\delta(G)\ge \frac{m-1}{2}$ and $\Delta(G) \ge \Delta(T)$, then G contains  a subgraph $T'\cong T$. 
\end{theorem}

\subsection{A connected triple}

Let $G$ be a graph. 
For a pair of vertices $u,v \in V(G)$, we denote by $p_G(u,v)$ the maximum number of internally vertex-disjoint $(u,v)$-paths in $G$.
Let $u$ and $v$ be nonadjacent vertices in $G$. 
A {\it $\{u,v\}$-separating set} is a subset $S$ of $V(G) \setminus \{u,v\}$ such that $u$ and $v$ belong to different components of $G-S$. 

\begin{theorem}[Menger's Theorem~\cite{BM08}]\label{thm:menger}
Let $G$ be a graph. For a pair of nonadjacent vertices $(u,v)$, the maximum number of pairwise internally vertex-disjoint $(u,v)$-paths is equal to the minimum cardinality of a $\{u, v\}$-separating set. 
\end{theorem}

For a vertex set $U\subseteq V(G)$, we let $\kappa_G(U) = \min \{p_G(u,v) \mid {u,v\in U} \text{ and }u\neq v \} $.
By Theorem~\ref{thm:menger}, for a positive integer $p$, $\kappa_G(U)\ge p+1$ implies that the vertices in $U-R$ lie in the same component of $G-R$ for every set $R\subseteq V(G)$ of at most $p$ vertices.
Note that $\kappa_G(V(G)) = \kappa(G)$.
A component of a graph is said to be {\it non-trivial} if it has at least two vertices.

Liu, Ying, and Hong~\cite{LYH24} introduced a special structure, called a {\it $p$-connected triple} (see Figure~\ref{fig:triple} for an illustration).

\begin{definition}[\!\!\cite{LYH24}]\label{def:contri}
Let $G$ be a graph and $S_1, S_2 \subseteq V(G)$.
Let $F$ be a nontrivial component of $G-(S_1 \cup S_2)$ ($F = G- (S_1 \cup S_2)$ is allowed).
For a positive integer $p$, $(S_1, S_2, F)$ is called a $p$-connected triple of $G$ if
\begin{itemize}
\item[(i)] $S_1 \cap S_2 = \emptyset$ and $|S_1 \cup S_2| \le 2p-1$,
\item[(ii)] $|N_G(v) \cap V(F)| \le p$ for any $v \in S_1$,
\item[(iii)] $\kappa_{G[S_1 \cup S_2\cup V(F)]} (S_2 \cup V(F)) \ge p+1$.
\end{itemize}
\end{definition}

\begin{figure}[!ht]
\centering
\includegraphics[page = 1,height=4cm]{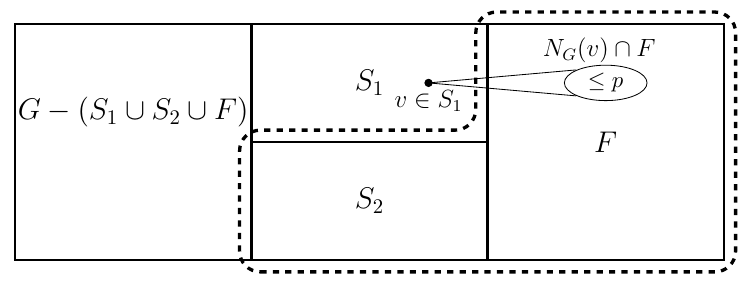}
\caption{A $p$-connected triple $(S_1,S_2,F)$ of $G$, where the region surrounded by dotted lines is related to the condition (iii) of Definition~\ref{def:contri}} \label{fig:triple}
\end{figure}

They also showed the existence of a $p$-connected triple as follows.

\begin{lemma}[\!\!\cite{LYH24}]\label{lem:contriexists}
Let $G$ be a graph with $\delta(G) \ge 2p$ and let $S_0 \subseteq V(G)$ with $|S_0| \le 2p-1$ for a positive integer $p$.
Then for any component $C$ of $G- S_0$, there exists a $p$-connected triple $(S_1, S_2, F)$ such that $V(F) \subseteq V(C)$.
\end{lemma}

\section{Proof of Theorems~\ref{thm:main_trianglefree} and \ref{thm:main_bip_g5}}

It is well-known that for triangle-free graphs, the following hold.

\begin{fact}\label{fact:trifree}
Let $G$ be a triangle-free graph.
Then the following are true.
\begin{itemize}
\item[\rm(a)] For any two adjacent vertices $u$ and $v$, $N_G(u)$ and $N_G(v)$ are disjoint. 
\item[\rm(b)] There is a vertex $v \in V(G)$ with $d_G(v) \le \frac{|V(
G)|}{2}$. 
\end{itemize}
\end{fact}

Moreover, we use the following classical result.
Let $G$ be a graph. 
For disjoint vertex sets $U$ and $W$ of $G$, a {\it matching} between $U$ and $W$ is an edge set $M \subseteq E(G)$ such that every edge has one endpoint in $U$ and the other in $W$, and no two edges in $M$ share a common vertex. 
For a vertex set $W\subseteq V(G)$, a matching $M$ in $G$ {\it saturates} $W$ if every vertex in $W$ is incident with an edge of $M$.

\begin{theorem}[Hall's Theorem~\cite{BM08}]\label{thm:hall}
For a bipartite graph $G$ with bipartition $(X,Y)$, $G$ has a matching which saturates $X$ if and only if $|N_G(S)| \ge |S|$ for all $S \subseteq X$. 
\end{theorem}

To prove Theorems~\ref{thm:main_trianglefree} and \ref{thm:main_bip_g5}, we present two useful lemmas.

\begin{lemma}\label{lem:matching}
    Let $G$ be a triangle-free graph with $\delta(G)\ge 2p$ for a positive integer $p$.
    Suppose that a $p$-connected triple $(S_1,S_2,F)$ of $G$ is chosen so that $|V(F)|$ is as small as possible.
    Then $|V(F)|>|S_1|$ and there exists a matching $M$ between $S_1$ and $V(F)$ that saturates $S_1$.
\end{lemma}
\begin{proof} 
Take an edge $uv$ of $F$, which exists since $F$ is nontrivial and connected.
Then $N_G(u)$ and $N_G(v)$ are disjoint by Fact~\ref{fact:trifree}(a).
Thus, since $F$ is a component of $G-(S_1\cup S_2)$,
\[
4p \le 2\delta(G) \le d_G(u)+ d_G(v) \le |S_1|+|S_2|+|V(F)|.\]
Hence, $|V(F)|\ge 4p-|S_1|-|S_2|\ge 2p+1 >|S_1|$.

We next prove the second assertion. Suppose not. 
Then, in view of Theorem~\ref{thm:hall}, there exists a nonempty $S\subseteq S_1$ such that $|N_G(S) \cap V(F)|< |S|$. 
Since $|V(F)|>|S_1|\ge|S|$, we can choose $K\subseteq V(F)$ so that
\[ 
N_G(S)\cap V(F) \subseteq K \quad \text{and} \quad |K|=|S|,
\]
and let $S_0=(S_1\setminus S)\cup S_2 \cup K$.
Then $|S_0|=|S_1|+|S_2|\le 2p-1$.
Moreover, since $F$ is a component of $G-(S_1\cup S_2)$ and $N_G(S)\cap V(F) \subseteq K$, we have $N_G(V(F)\setminus K) \subseteq S_0$.
Note that $1\le |K|=|S|<|V(F)|$.
There exists a component $C$ of $G-S_0$ with $V(C)\subseteq V(F)\setminus K$.
Applying Lemma~\ref{lem:contriexists} to $S_0$ and $C$, we see that there is a $p$-connected triple $(S_1',S_2',F')$ such that $V(F')\subseteq V(C)\subseteq V(F)\setminus K\subsetneq V(F)$, which contradicts the choice of $F$. 
Therefore $|N_G(S) \cap V(F)|\ge  |S|$ for every nonempty $S \subseteq S_1$.
Hence, by Theorem~\ref{thm:hall}, there is a matching $M$ between $S_1$ and $V(F)$ that saturates $S_1$.
\end{proof}

\begin{lemma}\label{lem:k-connected}
    Let $G$ be a $k$-connected triangle-free graph with $\delta(G)\ge 2p$ for positive integers $p$ and $k$ satisfying $p\ge k$.
    Suppose that $(S_1,S_2,F)$ is a $p$-connected triple of $G$ and there exists a matching $M$ between $S_1$ and $V(F)$ that saturates $S_1$. 
    Then $G-R$ is still $k$-connected for any set $R \subseteq V(F)\setminus V(M)$ of $p-k+1$ vertices. 
\end{lemma}

\begin{figure}[!ht]
\centering
\includegraphics[page = 2,height=6cm]{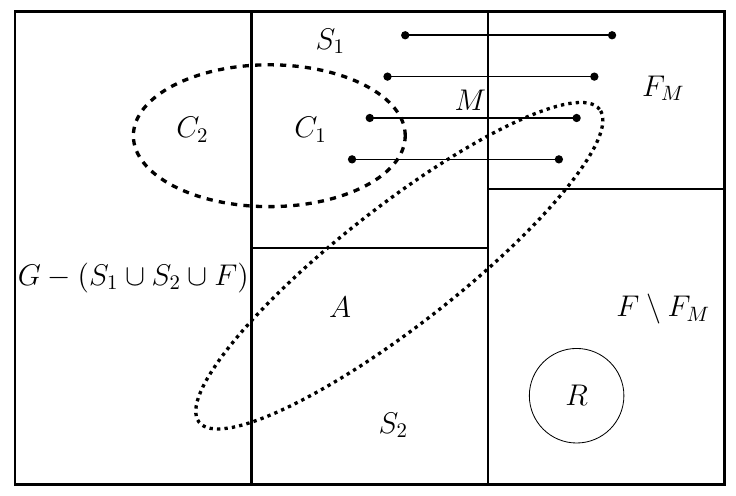}
\caption{An illustration for the proof of Lemma~\ref{lem:k-connected}} \label{fig:proof}
\end{figure}

\begin{proof}
Let $F_M=V(M) \cap V(F)$.
Suppose, to the contrary, that there exist a vertex set $R \subseteq V(F)\setminus V(M)$ with $|R|=p-k+1$, and a cut $A$ with $|A| \le k-1$ in $G-R$. (See Figure~\ref{fig:proof} for an illustration, where $C_1$ and $C_2$ will be defined later.)
Then $A \cup R$ is a cut in $G$ of order at most $p$.
We first observe that $(S_2 \cup V(F)) \setminus (A \cup R) \neq \emptyset$.
Note that $F$ is a nontrivial component of $G-(S_1\cup S_2)$.
Take an edge $uv$ of $F$.
Since the sets $N_G(u)$ and $N_G(v)$, which are disjoint by Fact~\ref{fact:trifree}(a), are contained in $S_1\cup S_2\cup V(F)$,  
\[|S_1 \cup S_2 \cup V(F)| \ge d_G(u)+d_G(v) \ge 2\delta(G) \ge 4p.\]
Thus $|S_2 \cup V(F)| \ge 4p - |S_1| \ge 4p-(2p-1) = 2p+1 > p \ge |A \cup R|$, so $(S_2 \cup V(F)) \setminus (A \cup R) \neq \emptyset$.
By Definition~\ref{def:contri}~(iii), $\kappa_G(S_2\cup V(F)) \ge \kappa_{G[S_1\cup S_2 \cup V(F)]}(S_2 \cup V(F)) \ge p+1$.
Therefore the remaining vertices of $S_2 \cup V(F)$ in $G-(A\cup R)$ must be contained in the same component.
Thus, there exists a component $C$ of $G- (A\cup R)$ such that $V(C) \cap (S_2 \cup V(F)) = \emptyset$.
We note that $N_G(V(C)) \subseteq A \cup R$.
Let
\[C_1 = V(C) \cap S_1, \quad C_2 = V(C) \setminus S_1, \quad A_1 = A \cap V(F), \quad A_2 = A \setminus V(F), \quad \text{and} \quad A_3 = N_G(C_1) \cap F_M. \]
Then, since the matching $M$ saturates $S_1$, $|A_3| \ge |C_1|$. Moreover, \[A_3 = N_G(C_1) \cap F_M \subseteq N_G(V(C)) \cap (V(F) \setminus R) \subseteq (A\cup R) \cap (V(F) \setminus R)= A\cap V(F) = A_1.\]
To the contrary, suppose that $C_2 \neq \emptyset$.
Then $N_G(C_2) \subseteq N_G(V(C))\cup C_1
\subseteq A \cup R \cup C_1$.
Since $N_G(V(F))\subseteq S_1 \cup S_2$ and $C_2 \cap (S_1 \cup S_2)=\emptyset$, $N_G(C_2)\cap V(F)=\emptyset$ and so $N_G(C_2) \subseteq (A \cup R \cup C_1)\setminus V(F)=(A\setminus V(F))\cup C_1=A_2 \cup C_1$.
Thus
\[ |N_G(C_2)| \le |A_2| + |C_1| \le |A_2|+|A_3| \le |A_2|+|A_1| = |A| \le k-1,\] which contradicts the $k$-connectivity of $G$.
Therefore $C_2 = \emptyset$ and so 
$V(C) = C_1$.
Since $G[V(C)]$ is triangle-free, there is a vertex $w \in V(C)$ such that $d_{G[V(C)]}(w) \le \frac{|V(C)|}{2}$ by Fact~\ref{fact:trifree}~(b).
Moreover, since $C$ is a component of $G-(A \cup R)$, we have $N_G(w)\subseteq N_{G[C]}(w)\cup (A\cup R)$.
Therefore
\[2p \le \delta(G) \le d_G(w) \le d_{G[C]}(w)+|A \cup R|\le \frac{|V(C)|}{2}+p \]
and so $|V(C)| \ge 2p$, which contradicts the fact that $V(C)=C_1 \subseteq S_1$ and $|S_1\cup S_2|\le 2p-1$.
Hence we have shown that $G-R$ is $k$-connected for any set $R \subseteq V(F)\setminus V(M)$ of $p-k+1$ vertices.
\end{proof}

Now we are ready to finish the proof.

\begin{proof}[Proof of Theorems~\ref{thm:main_trianglefree} and \ref{thm:main_bip_g5}] 
Let $T$ be a tree of order $m$ with bipartition $(X,Y)$, and $G$
be a  graph with $\delta(G)\ge 2k+2m+\beta-3$ where 
\[\beta=\begin{cases}
    m-1 & \text{if $G$ is a triangle-free graph,} \\
    \max\{|X|,|Y|\} & \text{if $G$ is a bipartite graph,}\\
    \max\{\frac{m-1}{2}, \Delta(T)\} & \text{if $g(G) \ge 5$.}
\end{cases}\]
If $m=1$, then Theorems~\ref{thm:main_trianglefree} and \ref{thm:main_bip_g5} are true by Theorem~\ref{thm:Chartrand}.
We assume $m \ge 2$. 
Let $p = k+m-1$. Then $\delta(G) \ge 2p+\beta-1\ge 2p$.
By Lemma~\ref{lem:contriexists}, there exists a $p$-connected triple $(S_1, S_2, F)$ of $G$.
We assume that $|V(F)|$ is as small as possible.
Then, by Lemma~\ref{lem:matching}, there is a matching $M$ between $S_1$ and $V(F)$ saturating $S_1$.
Let $F_M = V(M) \cap V(F)$ and $H=G[V(F) \setminus F_M]$.
Then $|F_M|=|S_1|$ and $F_M \subseteq N_G(S_1) \cap V(F)$.
Since $|V(F)|> |S_1|$ by Lemma~\ref{lem:matching}, $|V(H)|=|V(F) \setminus F_M|=|V(F)|-|F_M|=|V(F)|-|S_1|>0$
and so
$V(H) \neq \emptyset$.
Let $v$ be an arbitrary vertex of $H$.
Then, since $G$ has no triangles and $M$ is a matching between $S_1$ and $F_M$, $|N_G(v) \cap (S_1 \cup F_M)| \le |S_1|$.
Thus
\begin{align*}
d_H(v) &=
|N_G(v) \setminus (S_1 \cup S_2 \cup F_M)|  \\ &\ge \delta(G) - |N_G(v) \cap (S_1 \cup F_M)|  - |N_G(v) \cap S_2| \\
&\ge \delta(G) - (|S_1|+|S_2|)\\
& \ge (2p+\beta-1)-(2p-1)\\
&= \beta.
\end{align*}
Since $v$ was arbitrarily chosen, we have $\delta(H) \ge \beta$.
In all three cases, the corresponding previous results (Lemma~\ref{lem:tree_embedding}, Theorem~\ref{thm:tree_embedding_bipartite}, or Theorem~\ref{thm:g5}) guarantee a subtree $T' \cong T$ in $H$; note that $H$ inherits the bipartiteness or the girth condition from $G$, and that in the case $g(G) \ge 5$ we also have $\Delta(H) \ge \delta(H) \ge \beta \ge \Delta(T)$, as required in Theorem~\ref{thm:g5}.
Consequently, Lemma~\ref{lem:k-connected} implies that $G-V(T')$ is $k$-connected.
\end{proof}

\section{Concluding Remarks}

 Yang and Tian~\cite{YT25} proved that Conjecture~\ref{conj:bip} holds for odd paths, and Fujita~\cite{Fujita25} generalized this result by relaxing the bipartite condition to the triangle-free setting. Thus, we are led to ask whether Conjecture~\ref{conj:bip} can be extended to triangle-free graphs.

 \begin{question}\label{Q:tri}
 Let $k$ be a positive integer.
 For any tree $T$ with bipartition $(X,Y)$, does every $k$-connected triangle-free graph $G$ with minimum degree at least $k+ \max\{|X|,|Y|\}$, contain a subtree $T' \cong T$ such that $G- V(T')$ is still $k$-connected?
 \end{question}

More generally, existing results on tree embeddings in graphs of large girth indicate that the degree condition in our theorem may be further relaxed. For example, Jiang~\cite{Jiang01} showed that, for a tree $T$ of order $m$, a graph $G$ satisfying $g(G)\ge 2t+1$ and $\delta (G) \ge \max \{\frac{m-1}{t}, \Delta(T)\}$ has a subtree $T' \cong T$. From this, one can observe that the following is a direct consequence of our proof.

\begin{theorem}
Let $k$ and $t$ be positive integers with $t\ge 2$.
Let $G$ be a $k$-connected graph with $g(G)\ge 2t+1$ and $T$ be a tree of order $m$. 
If $\delta (G) \ge 2k+2m+ \max \{\frac{m-1}{t}, \Delta(T)\}-3$, then $G$ contains a subtree $T' \cong T$ such that $\kappa(G-V(T'))\ge k$. 
\end{theorem}
 
We finally mention the recent resolution of the Erd\H{o}s-S\'os conjecture, which states that every $n$-vertex graph with more than $\frac{(m-2)n}{2}$ edges contains every tree of order $m$.
The conjecture was recently proved by GPT-6 Astra (see Bloom~\cite{Bloom26} for an exposition), and Riordan and Scott~\cite{RS26} also characterized the extremal graphs. 
Thus, in the proof of Theorem~\ref{thm:main_trianglefree}, replacing Lemma~\ref{lem:tree_embedding} by the Erd\H{o}s--S\'os theorem with its extremal cases allows the bound to be lowered by one. 
Namely, it is easy to see the following: 
For a positive integer $k$, for a non-star tree $T$ of order $m$, every $k$-connected triangle-free graph $G$ with $\delta(G) \ge 2k+3m-5$ contains a subtree $T' \cong T$ such that $\kappa(G - V(T')) \ge k$. 
Following the proof of Theorem~\ref{thm:main_trianglefree} with $\beta=m-2$ for the case of triangle-free graphs, we obtain the graph $H$ with $\delta(H)\ge m-2$.
If $H$ were $T$-free, the Erd\H{o}s--S\'os theorem and the extremal characterization would force $H$ to be a vertex-disjoint union of copies of $K_{m-1}$, contrary to triangle-freeness since $m\ge4$.
The rest of the proof is unchanged.
For $T=K_{1,m-1}$ with $m\ge2$, one can verify that the same bound holds, although this requires a more careful analysis of the equality case in the degree estimates.
It would also be interesting in its own right to establish Conjecture~\ref{conj:mader} or answer Question~\ref{Q:tri} affirmatively for stars and arbitrary $k$.

\section*{Acknowledgement}
Hojin Chu was supported by a KIAS Individual Grant (CG101801) at Korea Institute for Advanced Study.
Shinya Fujita was supported by JSPS KAKENHI (23K03202).
Boram Park was supported by the National Research Foundation of Korea grant (No. RS-2025-00523206) funded by the Korea government (MSIT) and the New Faculty Startup Fund from Seoul National University.

\section*{Declaration of AI assistance}
The main results of this paper (Theorems~\ref{thm:main_trianglefree} and~\ref{thm:main_bip_g5}) and their proofs were developed by the authors without the use of generative AI.

\end{document}